\documentclass[leqno 12pts]{amsart}
\usepackage{amsmath, amsthm, amscd, amsfonts, amssymb, graphicx, color}
\usepackage[bookmarksnumbered, colorlinks, plainpages]{hyperref}
\usepackage{tikz}

\newtheorem{theorem}{Theorem}[section]
\newtheorem{lemma}[theorem]{Lemma}

\newtheorem{remark}[theorem]{Remark}

\newtheorem{folklore}[theorem]{Folklore}

\usepackage{comment}
\excludecomment{mysection}

\begin{document}
	
\title[Absolute compatibility]{Absolutely compatible pair of elements in a von Neumann algebra-II}
\author{Anil Kumar Karn}
	
\address{School of Mathematical Sciences, National Institute of Science Education and Research, HBNI, Bhubaneswar, P.O. - Jatni, District - Khurda, Odisha - 752050, India.}

\email{\textcolor[rgb]{0.00,0.00,0.84}{anilkarn@niser.ac.in}}
	
	
	
	
\begin{abstract}
	Let $A$ be a unital C$^*$-algebra with unity $1_A$. A pair of elements $0 \le a, b \le 1_A$ in $A$ is said to be \emph{absolutely compatible} if, $\vert a - b \vert + \vert 1_A - a - b \vert = 1_A.$ In this paper we provide a complete description of absolutely compatible pair of strict elements in a von Neumann algebra. The end form of such a pair has a striking resemblance with that of a `generic pair' of projections on a complex Hilbert space introduced by Halmos. 
\end{abstract}
\maketitle 

\section{Introduction}
Let $A$ be a C$^*$-algebra. A pair of elements $a, b \in A$ is said to be orthogonal, if $$a b = 0 = b a = a^* b = a b^*.$$ Orthogonal pair of positive elements play an important role in the theory of C$^*$-algebras. For example, it follows from the functional calculus that every self-adjoint element $a \in A_{sa}$ has a unique decomposition: $a = a^+ - a^-$ in $A^+$, where $a^+$ is algebraically orthogonal to $a^-$. Recently, the author proved an order theoretic characterization of algebraic orthogonality among positive elements of a C$^*$-algebra \cite{K18}. (Also see \cite{K14, K16}.) 

Orthogonal pairs of positive elements of norm $\le 1$ exhibit an interesting property. Let $A$ be a unital C$^*$-algebra with unity $1_A$. For a pair of elements $0 \le a, b \le 1_A$ in $A$, we have $a b = 0$ ($a$ is algebraically orthogonal to $b$) if and only if $a + b \le 1_A$ and $\vert a - b \vert + \vert 1_A - a - b \vert = 1_A.$ We isolate the later part and propose the following definition: A pair of elements $0 \le a, b \le 1_A$ in $A$ is said to be \emph{absolutely compatible}, (\cite{K18}), if $$\vert a - b \vert + \vert 1_A - a - b \vert = 1_A.$$ 
It was proved in \cite{K18} that a projection $p$ in a C$^*$-algebra $A$ is absolutely compatible with a positive element $a$ of $A$ with $\Vert a \Vert \le 1$ if, and only if, $a p = p a$. However, the two notions are distinct in general. The notion of absolute compatibility was introduced as an instrument to prove a spectral decomposition theorem in the context of ``absolute order unit spaces'' \cite{K18}. As an absolute order unit space is an order theoretic generalization of unital C$^*$-algebras \cite{K18}, absolute compatibility appears to be a significant property. 

Keeping this point of view, the author, along with Jana and Peralta, initiated a study of absolute compatibility in operator algebras \cite{JKP, JKP1}. Let $M$ be a von Neumann algebra and let $0 \le a \le 1_M$. We write 
$$s(a) := \sup \lbrace p \in \mathcal{P}(M): p \le a \rbrace$$ 
and
$$n(a) := \sup \lbrace p \in \mathcal{P}(M): p a = 0 \rbrace.$$
For $0 \le a \le 1_M$, we say that $a$ is \emph{strict} in $M$, if $s(a) = 0$ and $n(a) = 0$. In \cite{JKP}, it was proved that an absolutely compatible pair of positive elements in a von Neumann algebra has a (matricial) decomposition as a direct sum of commuting and `strict' elements. Let $0 \le a, b \le 1_M$ such that $a$ is absolutely compatible with $b$. Then there exist mutually orthogonal projections $p_1, p_2, s, n_1, n_2 \in \mathcal{P}(M)$ with $p_1 + p_2 + s + n_1 + n_2 = 1_M$ such that  
$$a = p_1 \oplus a_1 \oplus a_2 \oplus 0 \oplus a_3$$
and 
$$b = b_1 \oplus p_2 \oplus b_2 \oplus b_3 \oplus 0$$
with respect to $\{ p_1, p_2, s, n_1, n_2 \}$ with $a_2$ and $b_2$ are strict and absolutely compatible in $sMs$ \cite[Theorem 2.10]{JKP}. Thus the study of absolutely compatible pair of elements reduces to such strict pairs. It was further proved that 
\begin{theorem}\label{characterization in vN}\cite[Theorem 2.12]{JKP} Let $M$ be a von Neumann algebra, and let $a, b $ be elements in $[0, 1]_{M}$. Then
	$a$ is absolutely compatible with $b$ if and only if there exists a projection $p_1$ in $M$ so that $a$ and $b$ have matrix representations,
	say $a = \left( \begin{array}{cc} a_{11} & a_{12} \\  a_{12}^{\ast} &
	a_{22} \end{array} \right),$ and $b = \left( \begin{array}{cc} b_{11} &
	b_{12} \\  b_{12}^{\ast} & b_{22} \end{array} \right)$ with respect to the set
	$\{ p_1, 1-p_1 = p_2\}$ {\rm(}i.e., $a_{ij} = p_i a p_j$ and $b_{ij} = p_i b p_j${\rm)} satisfying:
	\begin{enumerate}
		\item[$(i)$] $a_{12} + b_{12} = 0$;
		\item[$(ii)$] $a_{12} a_{12}^{\ast}
		= (p_1 - a_{11}) (p_1 - b_{11})$;
		\item[$(iii)$] $a_{12}^{\ast} a_{12} = a_{22} b_{22}= b_{22} a_{22}$;
		\item[$(iv)$] $a_{12} = a_{11} a_{12} + a_{12} a_{22} = b_{11} a_{12} +
		a_{12} b_{22}$.
	\end{enumerate}
\end{theorem}
 Using these decompositions, a complete description of absolutely compatible pair of strict elements was given for the finite dimensional algebras $\mathbb{M}_n$ in the same paper \cite[Theorem 3.9]{JKP}. 

 In the present paper we continue the investigation and provide a complete description of absolutely compatible pair of strict elements in a von Neumann algebra. However, as expected, the finite dimensional (matricial) techniques used in \cite{JKP} fail to work. At this juncture, the author would like to thank Professor Kalyan B. Sinha for his suggestion to use polar decomposition. This idea works just perfectly. The main results of the paper are Theorems \ref{6} and \ref{7}.
 \begin{theorem}\label{6}
 	Let $M$ be a von Neumann algebra with the underlying Hilbert space $H$. Assume that $a, b \in [0, 1]_M$ be strict and absolutely compatible pair. Put $p = 1 - r(a \circ b)$. 
 	\begin{enumerate}
 		\item Then $p H$ is isometrically isomorphic to $(1 - p) H$. In particular, $H \equiv K \oplus K$, where $K = p H$. 
 		\item There exist strict elements $a_1, b_1 \in [0, p] \cap M$ with $a_1 b_1 = b_1 a_1$, $a_1 + b_1 \le p$ together with $p - (a_1 + b_1)$ strict in $p M p$; and a unitary $U: H \to K \oplus K$ such that 
 		$$a = U^* \begin{bmatrix} a_1 & (a_1 b_1)^{\frac{1}{2}} \\ (a_1 b_1)^{\frac{1}{2}} & p - a_1 \end{bmatrix} U \ \textrm{and} \ b = U^* \begin{bmatrix} b_1 & - (a_1 b_1)^{\frac{1}{2}} \\ - (a_1 b_1)^{\frac{1}{2}} & p - b_1 \end{bmatrix} U.$$
 	\end{enumerate}
 \end{theorem}
\begin{theorem}\label{7}
	Let $M$ be a von Neumann algebra with the underlying Hilbert space $H$. Assume that $a, b \in [0, 1]_M$ be strict and commuting pair such that $a^2 + b^2 \le 1$ with $1 - (a^2 + b^2)$ strict. Put $a_1 = \begin{bmatrix} a^2 & a b \\ a b & 1 - a^2 \end{bmatrix}$ and $b_1 = \begin{bmatrix} b^2 & - a b \\ - a b & 1 - b^2 \end{bmatrix}$. Then $a_1, b_1 \in [0, 1]_{M_2(M)}$ and $a_1$ is absolutely compatible with $b_1$.
\end{theorem}
 It is interesting and surprising to note that the end form of an absolutely compatible pair of strict elements bears a striking resemblance with that of a `generic pair' of projections on a complex Hilbert space (studied by Halmos \cite{H69}). 
 
 A pair of closed subspaces $M$ and $N$ of a Hilbert space $H$ is said to be in {\emph generic position}, if each of the following four subspaces of $H$: $M \cap N$, $M \cap N^{\perp}$, $M^{\perp} \cap N$ and $M^{\perp} \cap N^{\perp}$ are trivial. Let $P$ and $Q$ be the projections of $H$ on $M$ and $N$ respectively. In \cite{H69}, Halmos proved that if $M$ and $N$ are in generic position, then there exists a Hilbert space $K$, commuting, positive and invertible contractions $C$ and $S$ in $B(K)$ and a unitary operator $U: H \to K \oplus K$ such that 
 $$P = U^* \begin{bmatrix} C^2 & C S \\ C S & S^2 \end{bmatrix} U \ \textrm{and} \ Q = U^*  \begin{bmatrix} C^2 & - C S \\ - C S & S^2 \end{bmatrix} U.$$ 
 
 This observation further signifies the importance of a pair of strict absolutely compatible elements in an operator algebra. This paper is in the sequel of \cite{JKP}.

\section{The main results}
We obtain some basic results to prove Theorems \ref{6} and \ref{7}.
\begin{lemma}\label{1}
	Let $M$ be a von Neumann algebra with the underlying Hilbert space $H$ and let $a \in [0, 1]_M$ be strict. If $p \in \mathcal{P}(M)$, then, for the matrix representation $a = \begin{bmatrix} a_{11} & a_{12} \\ a_{12}^* & a_{22} \end{bmatrix}$ with respect to $\lbrace p, 1 - p \rbrace$, we have $a_{11} \in [0, p]$ and $a_{22} \in [0, 1 - p]$ are also strict in the von Neumann algebras $pMp$ and $(1 - p)M(1 - p)$ respectively.
\end{lemma} 
\begin{proof}
	Without any loss of generality, we may assume that $p \neq 0, p \neq 1$. First, we show that $r(a_{11}) = p$. Put $p - r(a_{11}) = p_0$. Then $p_0 a_{11} = 0$ so that 
	$$0 \le \begin{bmatrix} p_0 & 0 \\ 0 & 1 - p \end{bmatrix} \begin{bmatrix} a_{11} & a_{12} \\ a_{12}^* & a_{22} \end{bmatrix} \begin{bmatrix} p_0 & 0 \\ 0 & 1 - p \end{bmatrix} = \begin{bmatrix} 0 & p_0 a_{12} \\ a_{12}^* p_0 & a_{22} \end{bmatrix}.$$ 
	Thus $p_0 a_{12} = 0$ and consequently, 
	$$\begin{bmatrix} p - p_0 & 0 \\ 0 & 1 - p \end{bmatrix} a \begin{bmatrix} p - p_0 & 0 \\ 0 & 1 - p \end{bmatrix} = a.$$ 
	Now, it follows that $1 \le r(a) \le 1 - p_0$ so that $p_0 = 0$. Thus $r(a_{11}) = p$. Similarly, $r(a_{22}) = 1 - p$. Since $n(x) = 1 - r(x)$ for any $x \in [0, 1]_M$, we have $n(a_{11}) = 0$ and $n(a_{22}) = 0$.
	
	Next, as $a$ is strict, so is $1 - a = \begin{bmatrix} p - a_{11} & - a_{12} \\ - a_{12}^* & 1 - p - a_{22} \end{bmatrix}$. Thus, as above, $r(p - a_{11}) = p$ and $r(1 - p - a_{22}) = 1 - p$.   Therefore, $s(a_{11}) = 0$ and $s(a_{22}) = 0$. Hence $a_{11} \in [0, p]$ and $a_{22} \in [0, 1 - p]$ are strict in the von Neumann algebras $pMp$ and $(1 - p)M(1 - p)$ respectively.
\end{proof}
\begin{lemma}\label{2}
	Let $M$ be a von Neumann algebra with the underlying Hilbert space $H$ and let $a \in [0, 1]_M$ be strict. If $a x = 0$ for some $x \in M$, then $x = 0$. In particular, if $a \xi = 0$ for some $\xi \in H$, then $\xi = 0$.
\end{lemma} 
\begin{proof}
	Let $a x = 0$. Then $a \vert x^* \vert^2 = a x x^* = 0$ so that $a \vert x^* \vert = 0$ and consequently, $a r(\vert x^* \vert) = 0$. Thus $r(\vert x^* \vert) \le n(a) = 0$ so that $r(\vert x^* \vert) = 0$. Now, it follows that $x = 0$. 
	
	Next, if $a \xi = 0$, then $a p = 0$ where $p$ is the projection of $H$ on the span of $\xi$. Now, by the first step, $p = 0$ so that $\xi = 0$.
\end{proof}
\begin{lemma}\label{3}
	Let $M$ be a von Neumann algebra with the underlying Hilbert space $H$ and let $a, b \in [0, 1]_M$ be strict. 
	\begin{enumerate}
		\item Then $a^{\frac{1}{2}}$ is also strict in $M$. 
		\item If $a b = b a$, then $a b$ is also strict in $M$.
	\end{enumerate}
\end{lemma} 
\begin{proof}
	\begin{enumerate}
		\item Let $p = s(a^{\frac{1}{2}})$. Then $p = p a^{\frac{1}{2}} = a^{\frac{1}{2}} p$. Squaring, we get $p = p a = a p$ so that $p \le s(a) = 0$. Thus $p = 0$ and consequently, $s(a^{\frac{1}{2}}) = 0$. Also, $n(a^{\frac{1}{2}}) = n(a) = 0$. Therefore, $a^{\frac{1}{2}}$ is strict. 
		\item Let $a b = b a$. Then $a b \in [0, 1]_M$ with $a b \le a$. Thus $s(a b) \le a b \le a$ so that $s(a b) \le s(a) = 0$. Therefore, $s(a b) = 0$. Next, we have $a b n(a b) = 0$. As $a$ and $b$ are strict, a repeated use of Lemma \ref{2} yields that $n(a b) = 0$. Thus $a b$ is also strict.
	\end{enumerate}
\end{proof}
\begin{remark}
	Let $a \in [0, 1]_M$. Then $a$ is strict if and only if $a^2$ is strict,
\end{remark}
The following results are a compilation of operator algebra folklore which can easily be found in the literature. (See, for example, \cite[Section I.5.2]{B06}.)
\begin{folklore}\label{4}
	Let $M$ be a von Neumann algebra with the underlying Hilbert space $H$ and let $x \in M$. If $x = u \vert x \vert$ is the polar decomposition of $x$, then 
	\begin{enumerate}
		\item $u: (1 - r(\vert x \vert)) H \to r(\vert x \vert) H$ is an unitary. In particular, $u$ is a partial isometry in $M$. 
		\item $\vert x \vert = u^* x$. 
		\item $r(\vert x\vert) = u^* u$. 
		\item $\vert x^* \vert^k = u \vert x \vert^k u^*$ for all $k \in \mathbb{N}$. 
		\item $\vert x \vert^k = u^* \vert x^* \vert^k u$ for all $k \in \mathbb{N}$. 
	\end{enumerate}
\end{folklore} 
Now we prove the main results of the paper.

\begin{proof}[Proof of Theorem \ref{6}]
	\begin{enumerate}
		\item Put $r(a \circ b) = p_1$ so that $p = 1 - p_1$. Now, by Theorem \ref{characterization in vN}, $a$ and $b$ have matrix representations 
		$$a = \begin{bmatrix} a_{11} & a_{12} \\ a_{12}^* & a_{22} \end{bmatrix} \ \textrm{and} \ b = \begin{bmatrix} b_{11} & - a_{12} \\ - a_{12}^* & b_{22} \end{bmatrix}$$ 
		with respect to $\lbrace p_1, p \rbrace$ such that 
		\begin{enumerate}
			\item $a_{12} a_{12}^* = (p_1 - a_{11}) (p_1 - b_{11})$; 
			\item $a_{12}^* a_{12} = a_{22} b_{22}$; and 
			\item $a_{12} = a_{11} a_{12} + a_{12} a_{22} = b_{11} a_{12} + a_{12} b_{22}$.
		\end{enumerate} 
		Since $a$ and $b$ are strict, we have $p_1 \ne 0, p_1 \ne 1$. Also, by Lemma \ref{1},  $a_{11}$ and $b_{11}$ are strict in $p_1 M p_1$ and $a_{22}$ and $b_{22}$ are strict in $p M p$. Consider the polar decomposition 
		$$a_{12} = u \vert a_{12} \vert$$
		so that $U: u^* u H \to u u^* H$ is a unitary. By (a) and Folklore \ref{4}(4), we have 
		\begin{eqnarray*}
		(p_1 - a_{11}) (p_1 - b_{11}) &=& \vert a_{12}^* \vert^2 = u \vert a_{12} \vert^2 \\
		&=& u (a_{22} b_{22}) u^* \le u u^* \le p_1.
		\end{eqnarray*}
		Now, as $a_{11}$ and $b_{11}$ are strict elements in $[0, p_1]$ with $a_{11} b_{11} = b_{11} a_{11}$, by Lemma \ref{3}(2), we may conclude that $(p_1 - a_{11}) (p_1 - b_{11})$ is also a  strict element in $[0, p_1]$. Thus $r((p_1 - a_{11}) (p_1 - b_{11})) = p_1$. Thus $u u^* = p_1 = 1 - p$. In a similar way, by using (b) and Folklore \ref{4}(5), we may prove that $u^* u = p$. Thus $u: p H \to (1 - p) H$ is a unitary.
		\item By (b), we have $a_{22} b_{22} = \vert a_{12} \vert^2$ so that $a_{22}$, $b_{22}$ and $\vert a_{12} \vert$ commute with each other. Thus by (c), we get  
		\begin{eqnarray*}
		u \vert a_{12} \vert &=& a_{12} = a_{11} a_{12} + a_{12} a_{22} \\
		&=& a_{11} u \vert a_{12} \vert + u \vert a_{12} \vert a_{22} \\
		&=& a_{11} u \vert a_{12} \vert + u a_{22} \vert a_{12} \vert.
		\end{eqnarray*}
		In other words, 
		$$\vert a_{12} \vert (a_{11} u + u a_{22} - u)^* (a_{11} u + u a_{22} - u) = 0.$$ 
		By Lemma \ref{3}, $\vert a_{12} \vert$ is strict in $p M p$ so that, by Lemma \ref{2}, $a_{11} u + u a_{22} = u$. Therefore,
		$$p_1 = u u^* = a_{11} u u^* + u a_{22} u^* = a_{11} + u a_{22} u^*$$
		whence $u a_{22} u^* = p_1 - a_{11}$. In the same way we can show that $u b_{22} u^* = p_1 - b_{11}$. Now, put $a_{22} = a_1$, $b_{22} = b_1$ and $U = \begin{bmatrix} 0 & p \\ u^* & 0 \end{bmatrix}$. Then $a_1$ and $b_1$ are strict elements in $[0, p]$ and $U: H \to K \oplus K$ is a unitary. Also, the matrix multiplications yield that  
		$$U^* \begin{bmatrix} a_1 & (a_1 b_1)^{\frac{1}{2}} \\ (a_1 b_1)^{\frac{1}{2}} & p - a_1 \end{bmatrix} U = a \ \textrm{and} \ U^* \begin{bmatrix} b_1 & - (a_1 b_1)^{\frac{1}{2}} \\ - (a_1 b_1)^{\frac{1}{2}} & p - b_1 \end{bmatrix} U = b.$$ 
		Finally, we show that $p - a_1 - b_1$ is a strict element of $[0, p]$. Note that $a \circ b = \begin{bmatrix} a_{11} + b_{11} - p_1 & 0 \\ 0 & 0 \end{bmatrix}$. Thus $r(a_{11} + b_{11} - p_1) = r(a \circ b) = p_1$. Next, 
		\begin{eqnarray*}
		a_{11} + b_{11} - p_1 &=& p_1 - u a_1 u^* + p_1 - u b_1 u^* - p_1 \\
		&=& u (p - a_1 - b_1) u^*
		\end{eqnarray*}
		for $u p u^* = u u^* u u^* = p_1^2 = p_1$. Also, then 
		$$u^* (a_{11} + b_{11} - p_1)^k u = (p - a_1 - b_1)^k$$
		for all $k \in \mathbb{N}$. Since for any $x \in M^+$, $r(x)$ is limit of the increasing sequence $\lbrace (\frac{1}{n} + x)^{-1} x \rbrace$ in the strong operator topology, we may conclude that 
		$$r(p - a_1 - b_1) = u^* r(a_{1} + b_{11} - p_1) u = u^* p_1 u = p.$$ 
		Thus $n(p - a_1 - b_1) = 0$. Also 
		$$s(p - a_1 - b_1) \le p - a_1 - b_1 \le p - a_1$$
		so that 
		$$s(p - a_1 - b_1) \le s(p - a_1) = 0$$ 
		for $p - a_1$ is strict. Therefore, $p - a_1 - b_1$ is a strict element in $[0, p]$.
	\end{enumerate}	
\end{proof}

\begin{proof}[Proof of Theorem \ref{7}]
	Let $\xi, \eta \in H$. Then 
	\begin{eqnarray*}
	\vert \langle a b \eta, \xi \rangle \vert^2 &=& \vert \langle b \eta, a \xi \rangle \vert^2 \le \langle b^2 \eta, \eta \rangle \langle a^2 \xi, \xi \rangle \\ 
	&\le& \langle (1 - a^2) \eta, \eta \rangle \langle a^2 \xi, \xi \rangle.
	\end{eqnarray*}
	Thus we obtain that 
	\begin{eqnarray*}
	\left\langle \begin{bmatrix} a^2 & a b \\ a b & 1 - a^2 \end{bmatrix} \begin{bmatrix} \xi \\ \eta \end{bmatrix} \right\rangle &=& \langle a^2 \xi, \xi \rangle + \langle a b \eta, \xi \rangle + \langle a b \xi, \eta \rangle + \langle (1 - a^2) \eta, \eta \rangle \\
	&\ge& \left( \langle a^2 \xi, \xi \rangle^{\frac{1}{2}} - \langle (1 - a^2) \eta, \eta \rangle^{\frac{1}{2}} \right)^2 \ge 0.
	\end{eqnarray*}
	Therefore, $a_1 \ge 0$. Similarly, we can also show that $b_1 \ge 0$. Again, 
	$$1_{M_2(M)} - a_1 = \begin{bmatrix} 1 - a^2 & - a b \\ - a b & a^2 \end{bmatrix} \ \textrm{and} \ 1_{M_2(M)} - b_1 = \begin{bmatrix} 1 - b^2 & a b \\ a b & b^2 \end{bmatrix}$$ 
	so that $a_1, b_1 \in [0, 1]_{M_2(M)}$. 
	
	Now, we show that $a_1$ is strict. Let $p$ be a projection in $M_2(M)$ such that $p \le a_1$. Then $p = a_1 p$. For $\begin{bmatrix} \xi \\ \eta \end{bmatrix} \in p (H \oplus H)$, we have $a_1 \begin{bmatrix} \xi \\ \eta \end{bmatrix} = \begin{bmatrix} \xi \\ \eta \end{bmatrix}$ so that 
	$$\xi = a^2 \xi + a b \eta \quad \textrm{and} \quad \eta = a b \xi + (1 - a^2) \eta.$$
	By the second condition, we get $a^2 \eta = a b \xi$ so that $a (a\eta -b \xi) = 0$. Since $a$ is strict, by Lemma \ref{2}, we may conclude that $a \eta = b \xi$. Using this in the first relation, we get $\xi = a^2 \xi + b^2 \xi$ so that $(1 - a^2 - b^2) \xi = 0$. Since $1 - a^2 - b^2$ is also strict, invoking Lemma \ref{2} again we conclude that $\xi = 0$. But then, $a \eta = 0$ whence $\eta = 0$ as $a$ is also strict. Now, it follows that $p = 0$ so that $s(a_1) = 0$.
	
	Next, let $q$ be a projection in $M_2(M)$ such that $a_1 q = 0$. If $\begin{bmatrix} \xi \\ \eta \end{bmatrix} \in q (H \oplus H)$, we have $a_1 \begin{bmatrix} \xi \\ \eta \end{bmatrix} = 0$. Thus 
	$$a^2 \xi + a b \eta = 0 \quad \textrm{and} \quad a b \xi + (1 - a^2) \eta = 0$$
	and, as above we again conclude that $q = 0$. Therefore, $n(a_1) = 0$ too whence $a_1$ is strict. In a similar manner, we can conclude that $b_1$ is also strict. Finally, we show that $a_1$ is absolutely compatible with $b_1$. We have 
	$$(a_1 - b_1)^2 = \begin{bmatrix} a^2 - b^2 & 2 a b \\ 2 a b & b^2 - a^2 \end{bmatrix}^2 = \begin{bmatrix} (a^2 + b^2)^2 & 0 \\ 0 & (a^2 + b^2)^2 \end{bmatrix}$$
	so that $\vert a_1 - b_1 \vert = \begin{bmatrix} a^2 + b^2 & 0 \\ 0 & a^2 + b^2 \end{bmatrix}$. Also 
	\begin{eqnarray*}
	\vert 1_{M_2(M)} - a_1 - b_1 \vert &=& \left\vert \begin{bmatrix} a - a^2 - b^2 & 0 \\ 0 & a^2 + b^2 - 1 \end{bmatrix} \right\vert \\
	&=& \begin{bmatrix} 1 - a^2 - b^2 & 0 \\ 0 & 1 - a^2 - b^2 \end{bmatrix}.
	\end{eqnarray*}
	It follows that $\vert a_1 - b_1 \vert + \vert 1_{M_2(M)} - a_1 - b_1 \vert = 1_{M_2(M)}$ so that $a_1$ is absolutely compatible with $b_1$.
\end{proof}

\thanks{{\bf Acknowledgements}: The author is thankful to Antonio M. Peralta for introducing to him the notion of `strict' elements.}

\end{document}